\documentclass[11pt]{amsart}
\usepackage{amsmath,amsfonts,amsthm,amssymb,amscd}
\def\classification#1{\def\@class{#1}}
\classification{\null}

\DeclareFontFamily{OT1}{rsfs}{}
\DeclareFontShape{OT1}{rsfs}{n}{it}{<-> rsfs10}{}
\DeclareMathAlphabet{\mathscr}{OT1}{rsfs}{n}{it}

\newcommand{\eps}{\varepsilon}
\newcommand{\vers}{\rightarrow}

\DeclareMathOperator{\mo}{\,mod}

\newcommand{\Int}{{\mathbb Z}}

\newcommand{\Real}{{\mathbb R}}
\newcommand{\Com}{{\mathbb C}}

\newtheorem{prop}{Proposition}[section]
\newtheorem{thm}[prop]{Theorem}

\newtheorem{lem}[prop]{Lemma}

\numberwithin{equation}{section}
\title{Improving Roth's theorem in the primes}
\author{Harald Andr\'es Helfgott}
\address{H. A. Helfgott, Department of Mathematics, University of Bristol,
  Bristol BS8 1TW, United Kingdom}
\email{h.andres.helfgott@bristol.ac.uk}
\author{Anne de Roton}
\address{A.\ de Roton, Institut Elie Cartan, UMR 7502, Nancy-Universit\'e,
CNRS, INRIA, B.P. 239, 54506 Vandouevre-l\`es-Nancy cedex, France;
A.\ de Roton,
Pacific Institute for the Mathematical Sciences (CNRS UMI PIMS)
University of British Comlumbia
Vancouver BC V6T 1Z2
Canada}
\email{deroton@iecn.u-nancy.fr}
\begin{document}
\begin{abstract}
Let $A$ be a subset of the primes. Let $$\delta_P(N) = 
\frac{|\{n\in A: n\leq N\}|}{|\{\text{$n$ prime}: n\leq N\}|}.$$
We prove that, if \[\delta_P(N)\geq
C\ 
\frac{\log \log \log N}{(\log \log N)^{1/3}}\] for $N\geq N_0$, where
$C$ and $N_0$ are absolute constants, then $A\cap \lbrack 1,N\rbrack$ contains
a non-trivial three-term arithmetic progression.

This improves on Green's result \cite{Gr}, which
needs \[\delta_P(N) \geq C'
\sqrt{\frac{\log \log \log \log \log N}{\log \log \log \log N}}.\]
\end{abstract}
\maketitle
\section{Introduction}
\subsection{History and statement}
In 1953, K. Roth \cite{Ro}
proved that any subset of positive integers of positive 
density contains infinitely many non trivial three-term arithmetic progressions.
More precisely, his result is as follows. Given a set $A\subset \mathbb{Z}^+$, 
we define the {\em density} of $A\cap \lbrack 1, N\rbrack$ by
$\delta(N) =
\frac1N{|\{n\in A:\; n\leq N\}|}$.
(We write $|S|$ for the number of elements of a set $S$.)
 Roth proved that,
given any set of integers
$A\subset \mathbb{Z}^+$ such that $\delta(N) \geq {C}/{\log \log N}$
for some $N\geq N_0$ (where $C$ and $N_0$ are absolute constants) there
must be at least one non-trivial
three-term arithmetic progression in $A\cap \lbrack 1, N
\rbrack$. (By a {\em non-trivial} arithmetic progression we mean one
with non-zero modulus, i.e., $(a, a+d, a+2d)$ with $d\ne 0$.)

Much later, Heath-Brown \cite{HB} (1987) and Szemer\'edi \cite{Sz} (1990)
 improved this result by
showing that it is enough to require that $\delta(N) \geq C (\log{N})^{-c}$ for some small positive $c$.
By considering Bohr sets where previous arguments had used arithmetic progressions, Bourgain relaxed
the condition to  $\delta(N) \geq C \sqrt{{\log\log{N}}/{\log{N}}}$ in 
\cite{Bo2} (1999)
and to $\delta(N) \geq C {(\log\log{N})^2}{(\log{N})^{-2/3}}$ in 
\cite{Bo3} (2006).


Van der Corput proved \cite{vdC} that the primes contain infinitely many non trivial
3-term arithmetic progressions. The question then arises -- is Roth's theorem
true in the primes? That is -- must a subset of primes of positive relative
density\footnote{Given a subset $A$ of the set $P$ of all primes, we define the relative density
$\delta_P(N)$ of $A$ to be $\delta_P(N) = 
{|\{n\in A: n\leq N\}|}/{|\{\text{$n$ prime}: n\leq N\}|}$.
We are asking whether, given $A\subset P$ such that $\delta_P(N)>\delta_0$
($\delta_0>0$) for some sufficiently large $N$, the set $A$ contains
a non-trivial 3-term arithmetic progression.}
 contain a non-trivial 3-term arithmetic progression? 

In \cite{Gr}, B. Green showed that the answer is ``yes''.
He proved that, given any subset $A$ of the primes such that $A\cap \lbrack 1,
N\rbrack$ has relative density 
$\delta_P(N)\geq C ({\log \log \log \log \log N }/{\log \log \log \log N})^{1/2}$ for some $N\geq N_0$, where $C$ and $N_0$ are absolute constants,
there exists a 3-term arithmetic progression in $A$.

We prove the following result.
\begin{thm}\label{thm:main}
Let $A$ be a subset of the primes. Assume that $A\cap \lbrack 1, N\rbrack$ 
is of relative density
\[\delta_P(N) \geq C\ 
\frac{\log \log \log N}{(\log \log N)^{1/3}}\]
for some $N\geq N_0$, where $C$ and $N_0$ are absolute
constants. Then
$A$ contains a non-trivial 3-term arithmetic progression.
\end{thm}
In other words, we gain two logs over what was previously known. One of the
two logs gained is ultimately due to an
 enveloping use of a sieve; this idea is by now
familiar to the specialists, and, indeed, it will come into our proof via
a restriction theorem from \cite{GT} (based partially on work on sieves in 
\cite{Ra}). The other gain of a
log stems
from a more essential change in approach.

Our overall procedure is as follows.
The first step is to replace the characteristic function $a$ of
$A$ by a smoothed-out version $a_1$ whose Fourier transform is close to 
that $a$ (and thus, as can be easily shown, $a_1$ 
behaves like $a$ does when it comes
to the number of $3$-term progressions). This is much the same as in
\cite[\S 6]{Gr}; it is in accord with the general strategy (the ``uniformity
strategy'') described in \cite[\S 6]{Ta}. We then show that the $\ell_2$-norm
of $a_1$ is actually small enough that one can find a set $A'$
of large density in the integers such that $a_1$ is large on $A'$. 
This reduces the problem over the primes to the problem over the integers.

\subsection{Notation}
Let $N'$ be a positive integer. 
Let $f: \mathbb{Z}/N'\mathbb{Z} \rightarrow\Com$ be a function in 
$l^1(\Int/N' \Int)$. We define the Fourier transform of $f$ as the function 
$$\begin{array}{cccc}
\hat{f}:&\Int/N'\Int&\rightarrow& \Com\\
&b&\mapsto & \sum_{n\in\Int/N'\Int}f(n)e(-nb/N') ,
\end{array}$$
where we write $e(x)$ for $e^{2\pi i x}$.
 We write $\pi$ for the reduction map $\pi:\mathbb{Z} \to
\mathbb{Z}/N'\mathbb{Z}$. Given $x\in \mathbb{R}$, we define
$\{x\}$ to be the distance of $x$ to the nearest integer. 
We define $\|n\| = \{n/N'\}$; this works because $\{x\}$ depends only on 
$x \mo 1$.

Given a finite or countable set $S$, a function $f:S\to \mathbb{C}$
and a parameter $0<r<\infty$, we define the $\ell_r$-norm $|f|_r$ of $f$
by $|f|_r = \left(\sum_{x\in S} |f(x)|^r\right)^{1/r}$.
\subsection{Acknowledgements} 

Travel was funded by PIMS (Pacific Institute for the Mathematical Sciences)
and
EPSRC (Engineering and Physical Sciences Research Council).
H. A. Helfgott is supported in part by EPSRC grant EP-E054919/1;
A. de Roton is partially supported 
by PIMS and CNRS during the 2009--2010 calendar
year. Thanks are due to J. Bourgain, for his prompt reply to some questions,
to B. Green, for inviting both authors to
visit him at different times, and to I. Laba and T. Wooley, for their
constant help.

\section{From the primes to the integers}
\subsection{From the primes to the set $\{n: \text{$b + n M$ is prime}\}$}
Let us first show that we can focus on the intersection
of the primes with an arithmetic progression of large modulus, rather
than work on all the primes.

\begin{lem}\label{lem:ostrobo}
Let $\alpha, z$ be positive real numbers and $N$ be a large integer. We define $M=\prod_{p\leq z}p$. Let $A$ be a subset of the primes less than $N$ such that $|A|\geq \alpha N/\log{N}$. Then there exists some arithmetic progression $P(b)=\{b+nM: 1\leq n\leq N/M\}$ such that
$$|P(b)\cap A|\gg \alpha\frac{\log{z}}{\log{N}}\frac{N}{M}-\log{z},$$
where the implied constant is absolute. 
\end{lem}
\begin{proof}
If $(b,M)\ne 1$, the set $\{m\in P(b): \text{$m$ prime}\}$ is empty. 
Since the progressions $P(b)$ with $(b,M)=1$ are distinct, we have
$$\sum_{b:(b,M)=1}|A\cap P(b)|=|A|-|A\cap[1,M-1]|\geq \alpha \frac{N}{\log{N}}-M.$$
But $\left|\{b\leq M: (b,M)=1 \}\right|\sim {M}/{\log{z}}\sim e^z/\log{z}$. Therefore there exists some progression $P(b)$ such that 
\begin{equation*}
|A\cap P(b)|\gg
\left(\alpha\frac{N}{\log{N}}-M\right)\frac{\log{z}}{M}\gg\alpha\frac{\log{z}}{\log{N}}\frac{N}{M}
-\log{z}.
\end{equation*}
\end{proof}
The passage to an arithmetic progression $b+nM$ of large modulus is exactly
what Green and Tao \cite[p.\ 484]{GT2} call the ``$W$-trick'' (due to
Green's use of the letter $W$ for $M$ in \cite{Gr}). Green uses the fact
that such a passage removes all but the largest peaks in the Fourier transform
of the primes, whereas we simply use in a more direct way the fact that
the elements of $\{n:\text{$b+n M$ prime}\}$ are not forbidden from having 
small divisors. Of course, these are two
manifestations of the same idea.

Now, we fix $z=\frac{1}{3} \log N$, 
$M=\prod_{p\leq z} p$, and let $N'$ be the least prime larger
than $\left\lceil {2N}/{M}\right\rceil$.
(The requirement $N'> 
\left\lceil {2N}/{M}\right\rceil$
 will ensure that no new three-term arithmetic progressions
are created when we apply the reduction map $\pi:\mathbb{Z}\to
\mathbb{Z}/N'\mathbb{Z}$ to a set contained in $\lbrack 1, N/M\rbrack$.)
By Bertrand's postulate, $N' \ll {N}/{M}$.
Let $A$ be a subset of the primes less than $N$
such that $|A|\geq \alpha N/\log N$.
We assume $\alpha\geq (\log{N})N^{-1/2}$ (say) and obtain from 
Lemma \ref{lem:ostrobo} that there is an arithmetic progression $P(b)$
such that  $|P(b)\cap A|\gg {\alpha}({\log{z}}/{\log{N}})N'$. 
We define $A_0$ to be 
\begin{equation}\label{eq:deureq}
A_0 = \pi\left(\left\{n=\frac{m-b}{M}: m\in P(b)\cap A\right\}\right).
\end{equation}
This is a subset of 
$\pi(\{n\in [1,N']: \text{$b+nM$ is prime}\})$ 
satisfying \[|A_0|\gg \alpha \frac{\log{z}}{\log{N}} N'.\]
Our task is to show that there is a non-trivial three-term arithmetic progression
in $A_0\subset \mathbb{Z}/N'\mathbb{Z}$. It will follow immediately that
there is a non-trivial three-term arithmetic progression in $A\subset \mathbb{Z}$.

\subsection{From the set $\{n: \text{$b + n M$ is prime}\}$ to the
set of integers}

Let $a$ be the normalised characteristic
 function of $A_0$, i.e., $a=({\log{N}}/({N'\log{z}}))\mathbf{1}_{A_0}$. Fixing $\delta>0$ and $\varepsilon\in (0,1/4)$ to
be chosen later, define
$R:=\left\{x\in\Int/N'\Int : |\hat{a}(x)|\geq \delta\right\}\cup \{1 \}$ 
and the Bohr set
$$B:=\left\{n\in \mathbb{Z}/N'\mathbb{Z} : \forall x\in R, \|nx\|\leq\varepsilon \right\}.$$
We also define on $\mathbb{Z}/N'\mathbb{Z}$ the functions $\sigma=\frac{1}{|B|}\mathbf{1}_B$ and $a_1=a\ast\sigma$.  

To begin with, we remark that 
$|a|_1=({\log{N}}/({N'\log{z}}))|A_0|\gg \alpha$ and 
$|a_1|_1=|a|_1 |\sigma|_1= |a|_1$. Thus $|a_1|_1\gg \alpha$, i.e.,
$a_1$ is large in $\ell_1$-norm. We will later show that $a_1$ is small in
$\ell_2$-norm. These bounds on the $\ell_1$-norm and the $\ell_2$-norm will enable
us to find a large set of integers on which $a_1$ is $\gg \frac{1}{N}$. This
will enable us to reduce the problem for large subsets of the primes
to Roth's theorem for large subsets of the integers.

We first have to show that $a_1$ is ``close'' to $a$ in the sense that we care
about, namely -- we must show that $a_1$ is large on all three terms of
many three-term arithmetic progressions if and only if the same is true of $a$
(i.e., if and only if $A$ contains many three-term 
arithmetic progressions). More precisely,
our aim is to bound from above the quantity
\begin{equation}\label{eq:ameri}
\Delta = N'\cdot \left|\sum_{\text{$n_1$, $n_2$, $n_3$ in AP}}
a(n_1) a(n_2) a(n_3) - \sum_{\text{$n_1$, $n_2$, $n_3$ in AP}}
a_1(n_1) a_1(n_2) a_1(n_3)\right|\end{equation}
where the sums $\sum_{\text{$n_1$, $n_2$, $n_3$ in AP}}$ are over all
triples $(n_1,n_2,n_3)$ of elements of $\Int/N'\Int$
in arithmetic progression. Since 
$(n_1,n_2,n_3)$ is an arithmetic progression if and only if
$n_1 + n_3 = 2 n_2$,
\[\Delta = \left| \sum_m \hat{a}(-2 m) (\hat{a}(m))^2 -
\sum_m \hat{a}_1(-2 m) (\hat{a}_1(m))^2\right|, \]
as we can see simply by replacing all Fourier transforms by their definitions
and using the fact that $\sum_m e((n_1 + n_3 - 2 n_2) m/N') = 0$
when $n_1 + n_3 - 2 n_2 \ne 0$.

We will show that $\Delta$ is small, namely, $\Delta \ll \epsilon + \delta$.
First note that, since $a_1 = a\ast \sigma$ and so $\hat{a_1} = 
\hat{a} \hat{\sigma}$,
\[\Delta \leq \sum_m |\hat{a}(-2 m) - (\hat{a}(m))^2| |1 -
\hat{\sigma}(-2 m) (\hat{\sigma}(m))^2| .\]

For $x\in R$, since $\sigma$ is supported on $B$ and $\sum_{n}\sigma(n)=1$, 
we have 
\begin{align*}
|\hat\sigma(x)-1|&= \left|\sum_{n\in \Int/N'\Int}\sigma(n)e(nx/N')-1\right|\\
&= \left|\sum_{n\in \Int/N'\Int}\sigma(n)-1+\sum_{n\in \Int/N'\Int}\sigma(n)(
e(nx/N')-1)\right|\\
&\leq \sum_{n\in B}\sigma(n)|e(nx/N')-1| \ll\sum_{n\in B}\sigma(n)\|nx\|\ll\eps.
\end{align*}
Similarly, for $x\in R$,
\[|\hat\sigma(-2x)-1| \ll \sum_{n\in B} \sigma(n) \|-2 n x \| \ll 
\sum_{n\in B} \sigma(n) \eps = \eps\]
and so
\begin{equation}\label{eq:trenzas}|1 -
\hat\sigma(-2 x) \hat\sigma(x)^2| \ll \eps\end{equation}
for $x\in R$, i.e., when $|\hat{a}(x)|\geq \delta$.

Before we proceed further, we need to bound $\hat{a}$ in an average sense.
\begin{lem}\label{lem:dough}
For $p>2$,
 \begin{equation}\label{eq:jojo}
\sum_{m\in\Int/N'\Int}|\hat{a}(m)|^p \ll_p 1.\end{equation} 
\end{lem}
This is the same as \cite[Lemma 6.6]{Gr}; the only difference is
that our function $a$ was defined with a much larger modulus $M$
than in \cite{Gr}, and thus we must use a restriction theorem
for an upper-bound sieve, rather than a restriction theorem
for the primes (such as \cite[(4.39)]{Bo}).
\begin{proof}
Applying \cite[Prop.\ 4.2]{GT} with $F(n)=b+nM$ 
and $R=N'^{1/10}$, 
we obtain that,
for $p>2$ and any complex sequence $(b_n)_n$,
\begin{equation}\label{eq:chorra}
\sum_{m\in\Int/N'\Int}\left|\frac{1}{N'}\sum_{n=1}^{N'}b_n\beta(n)e(-mn/N')\right|^p\ll_p \left(\frac{1}{N'}\sum_{n=1}^{N'}|b_n|^2\beta(n)\right)^{p/2},
\end{equation}
where $\beta$ is an enveloping sieve function
with $R = N'^{1/10}$.
This means that, according to \cite[Prop.\ 3.1]{GT}, 
$\beta:\mathbb{Z}^+\to \mathbb{R}$ is
 a non-negative function satisfying the majorant property 
$$\beta(n)\gg \mathfrak{S}_F^{-1}\cdot \log{R}\cdot \mathbf{1}_{X_{R!}}(n)$$
with 
$$\mathfrak{S}_F=\prod_p\frac{\gamma(p)}{1-1/p},$$
$$\gamma(p)=\frac1p\left|\{n\in\Int/p\Int, (p,b+nM)=1 \}\right|=\begin{cases}(1-1/p) &\mbox{ if }p>z\\1&\mbox{ if }p\leq z\end{cases}$$ 
and
$$X_{R!}=\{n\in\Int: \forall d\leq R\;\; (b+nM,d)=1  \}.$$
In particular, for any integer $n\in A_0$, we have $n\in X_{R!}$ and 
$$\beta(n)\gg (\log{R})\prod_{p\leq z}(1-1/p)^{-1}\gg \frac{\log{N}}{\log{z}}.$$
We apply (\ref{eq:chorra}) to the sequence $(b_n)_n$ defined by
$$b_n=\begin{cases}\frac{1}{\beta(n)}a(n)&\mbox{ if }n\in A_0\\ 0 &\mbox{ otherwise}\end{cases}$$
and get
\begin{align*}
\sum_{m\in\Int/N'\Int}|\hat{a}(m)|^p&=(N')^p\sum_{m\in\Int/N'\Int}\left|\frac{1}{N'}\sum_{n\in\Int/N'\Int}a(n)e(-mn/N)\right|^{p}\\
&\ll_p (N')^{p/2}\left(\sum_{n\in\Int/N'\Int}\frac{1}{\beta(n)}a(n)^2\right)^{p/2}\\
&\ll_p (N')^{p/2}\left(\sum_{n\in\Int/N'\Int}\frac{\log{z}}{\log{N}}\left(\frac{\log{N}}{N'\log{z}}\right)a(n)\right)^{p/2}\\
&\ll_p \left(\sum_{n\in\Int/N'\Int}a(n)\right)^{p/2}\ll_p 1,
\end{align*}
since $a(n)$ was normalised so that $\sum_{n}a(n)\ll 1$.
\end{proof}

By H\"older's inequality and Lemma \ref{lem:dough}, we have
\[\sum_{m\in\Int/N'\Int}\left|\hat{a}(-2m)\hat{a}(m)^2\right|  \leq
\left(\sum_{m\in\Int/N'\Int} |\hat{a}(m)|^{5/2}\right)^{2/5}
\left(\sum_{m\in\Int/N'\Int} |\hat{a}(m)|^{10/3}\right)^{3/5}
\ll 1.\]

Hence, by \eqref{eq:trenzas},
\[\sum_{m: |\hat{a}(m)|\geq \delta}\left|\hat{a}(-2m)\hat{a}(m)^2\right|\left|1 - \hat{\sigma}(-2m)\hat{\sigma}(m)^2 \right|
\ll \eps .\]

On the other hand (again by H\"older, and again by Lemma \ref{lem:dough}),
\[\begin{aligned}
\sum_{m :|\hat{a}(m)|< \delta}&\left|\hat{a}(-2m)\hat{a}(m)^2\right|\left|1 - \hat{\sigma}(-2m)\hat{\sigma}(m)^2 \right|\leq
2\sum_{m: |\hat{a}(m)|< \delta}\left|\hat{a}(-2m)\hat{a}(m)^2\right|\\
&\leq 2\left(\sum_{m:|\hat{a}(m)|< \delta } |\hat{a}(m)|^{5/2}\right)^{2/5}
\left(\sum_{m:|\hat{a}(m)|< \delta} |\hat{a}(m)|^{10/3}\right)^{3/5}\\
&\leq 2\left(\sum_{m\in\Int/N'\Int} |\hat{a}(m)|^{5/2}\right)^{2/5}
\left(\delta^{5/3}\sum_{m\in\Int/N'\Int} |\hat{a}(m)|^{5/3}\right)^{3/5}
\ll \delta.\end{aligned}\]

Thus 
\begin{equation}\label{majDelta}
\Delta\ll(\eps+\delta).
\end{equation}

\subsection{An upper bound for the $\ell_2$-norm of $a_1$}\label{subs:qanpi}
Our aim in this subsection is to bound from above the $\ell_2$-norm of
of the function $a_1 = a \ast \sigma$. (This will later enable us to show
that $a_1$ is in some sense close to the characteristic function of a set of
large density in the integers.) We will prove that $|a_1|_2 \ll
 {1}/{\sqrt{N'}}$, where the implied constant is absolute.

Recall that we write $\pi$ for the reduction map $\pi:\mathbb{Z}\to
\mathbb{Z}/N'\mathbb{Z}$.
Given a function $f:\mathbb{Z}/N'\mathbb{Z}\to \mathbb{C}$, 
we can lift it to a function $\tilde{f}:\mathbb{Z}\to \mathbb{C}$
supported on the interval $\lbrack - (N'-1)/{2}, (N'-1)/{2}\rbrack$:
\[\tilde{f}(n) = \begin{cases} f(n \mo N') &\text{if 
$n\in \lbrack - \frac{N'-1}{2}, \frac{N'-1}{2}\rbrack$,}\\
0 &\text{otherwise.}\end{cases}\]
By the definition of $A_0$ and $a$, we see that $A_0\subset
\pi\left(\left\lbrack 1, (N'-1)/{2}\right\rbrack\right)$, and thus
 $a$ is supported on $\pi\left(\left\lbrack 1, (N'-1)/{2}
\right\rbrack\right)$. By the definition of
$R$, $B$ and $\sigma$ and the assumption
$\varepsilon<1/4$, we see that $\sigma$ is supported on 
$\pi\left(\left\lbrack -{N'}/{4}, {N'}/{4}
\right\rbrack\right)$. 
Thus $|a \ast \sigma|_2 = |\tilde{a}\ast \tilde{\sigma}|_2$.

By the definition of $a$, we have $0\leq \tilde{a}(n)\leq \lambda(n)$, 
where $\lambda:\mathbb{Z}\to \mathbb{R}$ is defined by 
\begin{equation}\label{eq:astron}\lambda(n)=\begin{cases}\frac{\log{N}}{N'\log{z}}&\mbox { if $1\leq n\leq N'$ and $b+nM$ is prime}\\0 &\mbox{ otherwise.} 
\end{cases}\end{equation}
Recall that $\sigma$ is non-negative, and thus $\tilde{\sigma}$ 
is non-negative. Hence
$|\tilde{a}\ast\tilde{\sigma}|_2\leq|\lambda\ast\tilde{\sigma}|_2$.
We conclude that
\[|a_1|_2 = |a\ast \sigma|_2 = 
|\tilde{a}\ast\tilde{\sigma}|_2\leq|\lambda\ast\tilde{\sigma}|_2 .\]
 It is thus our task to prove that
$|\lambda\ast \tilde{\sigma}|_2 \ll {1}/{\sqrt{N'}}$.

We proceed as follows:
\begin{equation}\label{eq:ostor}\begin{aligned}
 \sum_n\left|\tilde{\sigma}\ast\lambda(n)\right|^2
&=\sum_n\left|\sum_m\tilde{\sigma}(m)\lambda(n-m)\right|^2\\
&=\sum_{m_1}\sum_{m_2}\tilde{\sigma}(m_1)\tilde{\sigma}(m_2)
\sum_n\lambda(n+m_1)\lambda(n+m_2), \end{aligned}\end{equation}
where we recall that $\sigma(m) = \sigma(-m)$ (by the definition of $B$
and $\sigma$).
\begin{lem}\label{lem:cogot}
Let $\lambda$ be as in (\ref{eq:astron}).
Then, for any integers $m_1$, $m_2$,
\begin{equation}\label{sieveconj}
\sum_n\lambda(n+m_1)\lambda(n+m_2)\ll\left\{\begin{array}{ccc}\log{N}/(N'\log{z})&\mbox{ if }&m_1=m_2,\\ \frac1{N'}\prod_{p|(m_1-m_2),\; p>z}\frac{p}{p-1}&\mbox{ if }&m_1\not=m_2,\end{array}\right.
\end{equation}
where the implied constant is absolute.
\end{lem}
\begin{proof}
The case $m_1=m_2$ follows from Brun-Titchmarsh:
\begin{align*}
\sum_n\lambda^2(n+m)&=\left(\frac{\log{N}}{N'\log{z}}\right)^2\left|\{m\leq n\leq N'+m : b+(n-m) M\mbox{ is prime }\}\right|\\
&\ll \left(\frac{\log{N}}{N'\log{z}}\right)^2\frac{N'M}{\varphi(M)\log{N'}}\\
&\ll \left(\frac{\log{N}}{N'\log{z}}\right)^2\frac{N'}{\log{N'}}
\prod_{p|M}(1-1/p)^{-1}\\
&\ll \frac{\log{N}}{N'\log{z}}.
\end{align*}

 To obtain the case $m_1\ne m_2$, we will use a result based on Selberg's sieve.
(This is a familiar type of application of upper-bound sieves, similar to the
proof that the number of twin primes up to $N$ is at most a constant times its
conjectured value.) It is clear that
$\sum_n\lambda(n+m_1) \lambda(n+m_2)$ equals 
$\left({\log{N}}/(N'\log{z})\right)^2$ times
\begin{equation}\label{eq:hugue}
\left|\{1\leq n\leq N' : \text{$b+nM$ and $b+(n+m_2-m_1)M$ are primes}\}\right|.
\end{equation}
By \cite[Thm.\ 5.7]{HR}, 
\[(\ref{eq:hugue}) \ll \prod_p \left(1-\frac{\rho(p) -1}{p-1}\right)
\left(1 - \frac{1}{p}\right)^{-1} \frac{N'}{(\log N')^2} ,\]
where the implied constant is absolute. (We are implicitly using the fact
that $\log M\ll \log N'$, and thus the term in the third line of
\cite[(8.3)]{HR} is $= 1+o(1)$.) Here $\rho(p)$ is the number of solutions
$x\in \mathbb{Z}/p\mathbb{Z}$ to
\[(b + x M) (b + (x+ m_2 - m_1) M) \equiv 0 \mo p\]
for $p$ prime. It is easy to see that $\rho(p)=0$ if $p|M$ (i.e.,
if $p\leq z$), $\rho(p)=1$ if $p>z$ and $p|(m_2-m_1)$, and
$\rho(p)=2$ if $p>z$ and $p\nmid (m_2-m_1)$. Hence
\[\begin{aligned}
(\ref{eq:hugue}) &\ll \prod_{p\leq z} \left(1 - \frac{1}{p}\right)^{-2}
\mathop{\prod_{p>z}}_{p|m_1-m_2} \left(1 - \frac{1}{p}\right)^{-1}
\frac{N'}{(\log N')^2}\\
&\ll 
\mathop{\prod_{p>z}}_{p|m_1-m_2}\frac{p}{p-1}
\frac{N' (\log z)^2}{(\log N')^2}.
\end{aligned}\]
The statement follows.
\end{proof}

Let us now evaluate the last line of (\ref{eq:ostor}), with Lemma \ref{lem:cogot} in hand.
The contribution of the diagonal terms ($m_1=m_2$) in (\ref{eq:ostor})
is $\ll {\log N}/(|B|{N'\log{z}})$.
The contribution of the non-diagonal terms ($m_1\ne m_2$) is
\begin{equation}\label{eq:monst}\begin{aligned}
&\ll \frac1{N'} \mathop{\sum_{m_1} \sum_{m_2}}_{m_2\not=m_1}
\tilde{\sigma}(m_1)\tilde{\sigma}(m_2) 
\mathop{\prod_{p>z}}_{p|m_1-m_2} \frac{p}{p-1}.
\end{aligned}\end{equation}
Recall that $\tilde{\sigma}$ is supported on $\left\lbrack -{N'}/{4},
{N'}/{4}\right\rbrack$, and thus $|m_2-m_1|\leq {N'}/{2} < N'$ whenever
$\tilde{\sigma}(m_1) \tilde{\sigma}(m_2) \ne 0$. 

Now, a non-zero integer
$m$ with $|m|\leq N'$ 
cannot have more than ${\log N'}/{\log z}$ prime factors
$p>z$. Since $x\mapsto {x}/({x-1})$ is decreasing on $x$, this means that
\[\mathop{\prod_{p>z}}_{p|m} \frac{p}{p-1} \leq 
\left(\frac{z}{z-1}\right)^{(\log N')/(\log z)}.\]
Now $({z}/({z-1}))^{z} \ll 1$ (because $\lim_n (1+1/n)^n = e$) and
$$\frac{\log N'}{\log z} \ll \frac{\log N}{\log \log N} < \log N \ll z.$$
Hence
\[\mathop{\prod_{p>z}}_{p|m} \frac{p}{p-1} \ll 1\]
for any $m\ne 0$ with $|m|\leq N'$.
Thus
\[(\ref{eq:monst}) \ll \frac{1}{N'} 
\mathop{\sum_{m_1} \sum_{m_2}}_{m_2\not=m_1} \tilde{\sigma}(m_1)\tilde{\sigma}(m_2) 
\ll \frac{1}{N'}.\]

Putting everything together, we conclude that
\[ \sum_n\left|\tilde{\sigma}\ast\lambda(n)\right|^2 \ll
\frac{1}{N'} \left(\frac{\log N}{|B|\log{z}} + 1\right).\]
The right side is $\ll {1}/{N'}$ as long as $|B|\gg \log N/\log z$.

Now, as is well-known (see, e.g., \cite[Lem.\ 4.20]{TV}),
\[|B| \gg  \eps^r N',\]
where $r=|R|$. (The proof of this is a simple pigeonhole argument.)
 Since by \eqref{eq:jojo} we have
$\sum_m |\hat{a}(m)|^{5/2}  \ll 1$, we know that 
that the set of $x\in \mathbb{Z}/N' \mathbb{Z}$ with $|\hat{a}(x)|\geq \delta$
has at most $\ll \delta^{-5/2}$ elements. Thus, $r\ll \delta^{-5/2}$.

Hence all that we need for $|B| \geq \log N/\log z$ to hold is that
$\eps^{\delta^{-5/2}}\geq N^{-1/2}$ (say). In other words, we need
$|\log \eps| \cdot \delta^{-5/2} \leq \frac{1}{2} \log N$. We will recall
that we need to satisfy this condition at the end.

\subsection{Extracting a dense set from $a_1$}

We now have a function $a_1:\mathbb{Z}/N'\mathbb{Z}\to \mathbb{R}_0^+$ 
of $\ell_2$ norm $\ll {1}/{\sqrt{N'}}$. 
Its $\ell_1$ norm is $\gg \alpha$, where $\alpha$
is the density of our original set $A$ on the primes.
We must show that there is a large set on which $a_1$ is large.

\begin{lem}
Let $S$ be a set with $N'$ elements.
Let $a:S\vers\Real_0^+$ and $0<\alpha<1$ be such that
\begin{enumerate}
\item $\|a\|_1\geq \alpha$;
\item $\|a\|_2^2 \leq c/N'$.
\end{enumerate}
Then there exists a subset $A'$ of $S$ such that
\begin{enumerate}
\item $|A'| \geq \alpha^2N'/(4c)$;
\item $\forall n\in A'$, $a(n)\geq \alpha/2N'$.
\end{enumerate}
\end{lem}

\begin{proof}
If $A'=\{n : a(n)\geq \alpha/(2N')\}$, then
\begin{align*}
\alpha&\leq \sum_n a(n)\leq \frac{\alpha}{2N'}(N'-|A'|)+\sum_{n\in A'}a(n)\\
&\leq  \frac{\alpha}{2N'}(N'-|A'|)+\sqrt{|A'|}\sqrt{\frac{c}{N'}},
\end{align*}
by $\|a\|_2\leq c/N'$ and Cauchy's inequality.
In other words, $f(\sqrt{|A'|})\leq 0$, where
$f(x) = \frac{\alpha}{2 N'} \left(x^2 - 2 \frac{\sqrt{c N'}}{\alpha} x + N'\right)$. Completing the square, we see that $f(x)\leq 0$ implies
$x\geq \frac{\sqrt{c N'}}{\alpha} - \sqrt{\left(\frac{c}{\alpha^2}-1\right) N'}$.
Hence
\[|A'|\geq N' \cdot \left(\frac{ \sqrt{c}}{\alpha} - \sqrt{
\frac{c}{\alpha^2} - 1}\right)^2 \geq \frac{\alpha^2 N'}{4 c} .\]
\end{proof}

We apply Lemma 2 to $a_1$ with the bound $\|a_1\|_2^2\leq c/N'$ being provided 
by our work in \S \ref{subs:qanpi}.
We get a subset $A'$ of $\mathbb{Z}/N'\mathbb{Z}$ 
such that $|A'|\geq {\alpha^2 N'}/(4c)\gg\alpha^2 N'$ and 
$a_1(n)\geq {\alpha}/(2N')$ for every $n\in A'$.
Hence
\begin{equation}\label{eq:jura}
\sum_{m,d}a_1(m)a_1(m+d)a_1(m+2d)\geq \frac{\alpha^3 Z}{8N'^3},\end{equation}
where  $Z$ is the number of 3-term arithmetic progressions in $A'$.

\begin{lem}\label{lem:hono}
Let $A'\subset \mathbb{Z}/N'\mathbb{Z}$, where $N'$ is a prime. 
Assume $|A'|\geq \eta N'$, $\eta>0$.
 The number of 3-term arithmetic progressions in $A'$ is then at least
\[
 \frac{\eta N'^2}{
c_0 \exp\left(c_1 \eta^{-3/2} (\log(1/\eta))^3\right)},
\]
where $c_0$ and $c_1$ are absolute constants.
\end{lem}
\begin{proof}
We will proceed much as in \cite[Lem.\ 6.8]{Gr}; the basic argument goes
back to Varnavides \cite{Va}. 
Bourgain's best result on three-term arithmetic progressions in the integers
\cite[Thm.\ 1]{Bo3} states that, for given $L$ and $\eta \gg
{(\log \log L)^2}{(\log L)^{-2/3}}$, every subset of $\{1,2,\dotsc,L\}$
with $\geq \eta L$ elements contains at least one non-trivial three-term
arithmetic progression. This can be rephrased as follows: 
  there are constants $c_0$ and $c_1$ such that, if 
$L\geq c_0 \exp\left(c_1 \eta^{-3/2} (\log(1/\eta))^3\right)$, 
then any subset of $\{1,\cdots,L\}$ of density at least $\eta/2$ contains a non-trivial three-term arithmetic progression.
(Here we are simply expressing $L$ in terms of the density, rather than
the density in terms of $L$.)

It follows that, given an arithmetic progression $S_{a,d} = \{a+ d, a+2 d, a+3d,
\dotsc, a+Ld\}$ in $\mathbb{Z}/N'\mathbb{Z}$ ($a,d\in \mathbb{Z}/N'\mathbb{Z}$,
$d\ne 0$, $L\leq N'$) 
whose intersection with $A'$ has at least $({\eta}/{2}) L$
elements, there is at least one non-trivial three-term arithmetic progression
in $A'\cap S \subset \mathbb{Z}/N'\mathbb{Z}$. (Note that there is no
need for the progression $S$ to be the reduction mod $N'$ of a progression in
the integers $\{1,2,\dotsc,N'\}$; the argument works regardless of this.)
If we consider all arithmetic progressions of length $L$ and 
given modulus $d\ne 0$
in $\mathbb{Z}/N'\mathbb{Z}$,
we see that each element of $A'$ is contained in exactly $L$ of them.
Hence, $\sum_a |S_{a,d} \cap A'| = L |A'| \geq \eta N' L$, and so
(for $d\ne 0$ fixed)
$|S_{a,d}\cap A'| \geq ({\eta}/{2}) L$ for at least $({\eta}/{2}) N'$
values of $a$. Varying $d$, we get that $|S_{a,d}\cap A'|\geq ({\eta}/{2})
L$ for at least $({\eta}/{2}) N' (N'-1)$ arithmetic progressions $S_{a,d}$.
By the above, each such intersection $S_{a,d}\cap A'$ contains at least
one non-trivial three-term arithmetic progression.

Each non-trivial three-term arithmetic progression $a_1, a_2, a_3$
 in $\mathbb{Z}/N'\mathbb{Z}$
can be contained in at most $L (L-1)$ arithmetic progressions
$\{a+d,a+2d,\dotsc,a+Ld\}$ of length $L$ (the indices of $a_1$ and $a_2$ in the 
progression of length $L$ determine the progression). Hence, when we count
the three-term arithmetic progressions coming from the intersections 
$S_{a,d}\cap A'$, we are counting each such progression at most $L (L-1)$ times.
Thus we have shown that $A'$ contains at least
\[\frac{\eta}{2} \frac{N' (N'-1)}{L (L-1)}\geq 
\frac{\eta}{2} \frac{N'^2}{L^2}\]
distinct non-trivial three-term arithmetic
progressions for \[L =
\left\lceil c_0 \exp\left(c_1 \eta^{-3/2} (\log(1/\delta))^3\right)
\right\rceil,\] 
provided that $L\leq N'$. If $L>N'$, the bound in the statement of the lemma
is trivially true (as there is always at least one trivial three-term arithmetic
progression in $A'$.)
\end{proof}

From (\ref{eq:jura}) and Lemma \ref{lem:hono}, we conclude that
\begin{equation}\label{eq:iturb}\begin{aligned}
\sum_{m,d}a_1(m)a_1(m+d)a_1(m+2d)&\geq 
\frac{\alpha^3}{8 N'^3} \frac{\alpha^2}{8 c} 
\frac{(N')^2}{c_0 \exp(c_1 \left(
\alpha^2/4 c)^{-3/2} (\log(4c/\alpha^2))^3\right)}\\
&\geq \frac{1}{N'} \frac{1}{c_2 \exp\left(
c_3 \alpha^{-3} (\log (1/\alpha))^3\right)} ,\end{aligned}\end{equation}
where $c_2, c_3>0$ are absolute constants.
\section{Conclusion}
Assume that $A$ contains no non-trivial three-term arithmetic progressions.
Then $A_0$ (defined in (\ref{eq:deureq}))
contains no non-trivial three term arithmetic progressions, and so
\[\sum_{m,d}a(m)a(m+d)a(m+2d)
= \sum_{m,d}a(m)^3\ll \left(\frac{\log{N}}{N'\log z}\right)^2.\]
We also have
\[\Delta=
N' \left|\sum_{m,d}a(m)a(m+d)a(m+2d) -
\sum_{m,d}a_1(m)a_1(m+d)_1(m+2d)\right| \ll (\eps + \delta),\]
by the definition (\ref{eq:ameri}) and (\ref{majDelta}).
Lastly, we have just shown that
\[
\sum_{m,d}a_1(m)a_1(m+d)a_1(m+2d)\geq
\frac{1}{N'} \frac{1}{c_2 \exp\left(
c_3 \alpha^{-3} (\log (1/\alpha))^3\right)} 
\]
(see (\ref{eq:iturb})). We conclude that
\begin{equation}\label{eq:fjul}
\frac{1}{c_2 \exp\left(
c_3 \alpha^{-3} (\log (1/\alpha))^3\right)} \ll \eps + \delta 
+ \frac{1}{N'} \left(\frac{\log N}{\log z}\right)^2 .\end{equation}

Recall that $z = ( \log N)/3$.
There are constants $c_4$, $c_5$ such that, for
$$\delta= \eps =  \frac{1}{c_4} 
\exp\left(-c_5\alpha^{-3} \log^3(1/\alpha)\right),$$ we get a contradiction with (\ref{eq:fjul}),
provided that $N$ is larger than an absolute constant and
 $\alpha\geq (\log{N})^{-1/4}$, say. These values of $\delta$ and $\eps$ 
satisfy
$|\log \eps| \delta^{-2.5} \leq ( \log N)/2$ as long as
$$\left(\log(c_4) + c_5 \alpha^{-3} \log^3(1/\alpha)\right)\cdot
c_4^{2.5} 
\exp\left(2.5 c_5\alpha^{-3} \log^3(1/\alpha)\right) \leq \frac{1}{2} \log N.$$ 
Therefore we have a contradiction if $\alpha\geq C
{\log\log\log{N}}\left(\log\log{N}\right)^{-1/3}$, where $C$ is a large enough constant and $N$ is larger than an absolute constant.
Theorem \ref{thm:main} is thereby proven.


\begin{thebibliography}{GT2}
\bibitem[Bo]{Bo} Bourgain, J., On $\Lambda(p)$-subsets of squares, Israel J.\
  Math. {\bf 67} (1989), no.\ 3, 291--311.
\bibitem[Bo2]{Bo2} Bourgain, J., On triples in arithmetic
progression, {\em Geom. Funct. Anal.} {\bf 9} (1999), no.\ 5, 968--984.
\bibitem[Bo3]{Bo3} Bourgain, J., Roth's theorem on progressions revisited,
\bibitem[Gr]{Gr} Green, B., Roth's theorem in the primes, {\em
  Ann. of Math. (2)} {\bf 161} (2005), 1609--1636.
\bibitem[GT]{GT} Green, B., and T. Tao, Restriction theory of the Selberg
sieve, with applications, {\em Journal de th\'eorie des nombres de Bordeaux}
{\bf 18} (2006), no.\ 1, p.\ 147--182.
\bibitem[GT2]{GT2} Green, B., and T. Tao, The primes contain arbitrarily
long arithmetic progressions, {\em Ann. of Math. (2)} {\bf 167}
(2008), no.\ 2, 481--547.
\bibitem[HR]{HR} Halberstam, H., and H. E. Richert, {\em Sieve methods},
Academic Press, 1974.
\bibitem[HB]{HB} Heath-Brown, D. R., Integer sets containing no arithmetic
progressions, {\em J. London Math. Soc. (2)} {\bf 35} (1987), no.\ 3,
385--394.
\bibitem[Ra]{Ra} Ramar\'e, O., On Snirel'man's constant, {\em Ann.
Scu. Norm. Pisa} {\bf 22} (1995), 645--706.
\bibitem[Ro]{Ro} Roth, K. F., On certain sets of integers,
{\em J. London Math. Soc.} {\bf 28} (1953), 104--109.
\bibitem[Sz]{Sz} Szemer\'edi, E., Integer sets containing no arithmetic 
progressions, {\em Acta Math. Hungar.} {\bf 56} (1990), no. 1--2, 155--158.
\bibitem[Ta]{Ta} Tao, T., Arithmetic progressions and the
primes, {\em Collect. Math.} {\bf 2006},  Vol.\ Extra, 37--88. 
\bibitem[TV]{TV} Tao, T., and V. Vu, {\em Additive combinatorics},
Cambridge University Press, 2006.
\bibitem[vdC]{vdC} van der Corput, J. G., \"Uber Summen von Primzahlen und
Primzahlquadraten, {\em Math. Ann.} {\bf 116} (1939), 1–-50.
\bibitem[Va]{Va} Varnavides, P., On certain sets of positive density,
J.\ London Math.\ Soc. {\bf 34} (1959) 
358--360.
\end{thebibliography}
\end{document}